\documentclass[11pt]{article}
\usepackage{amsmath,amssymb,amsfonts,amsthm,color}
\usepackage[lmargin=30mm,rmargin=30mm,tmargin=30mm,bmargin=30mm]{geometry}
\renewcommand{\baselinestretch}{1.15}
\usepackage[numbers,sort&compress]{natbib}
\usepackage[colorlinks=true,citecolor=black,linkcolor=black,urlcolor=blue]{hyperref}

\newcommand{\msn}[1]{MR:\,\href{http://www.ams.org/mathscinet-getitem?mr=MR#1}{#1}}
\newcommand{\MSN}[2]{MR:\,\href{http://www.ams.org/mathscinet-getitem?mr=MR#1}{#1}}
\newcommand{\doi}[1]{doi:\,\href{http://dx.doi.org/#1}{#1}}
\newcommand{\cs}[1]{CiteSeer:\,\href{http://citeseerx.ist.psu.edu/viewdoc/summary?doi=#1}{#1}}

\newtheorem*{WLHC}{Weak List Hadwiger Conjecture}
\newtheorem*{LHC}{List Hadwiger Conjecture}
\newtheorem*{HC}{Hadwiger Conjecture}
\newtheorem*{WHC}{Weak Hadwiger Conjecture}

\theoremstyle{plain}
\newtheorem{theorem}{Theorem}
\newtheorem{lemma}[theorem]{Lemma}

\DeclareMathOperator{\col}{col}
\newcommand{\floor}[1]{\ensuremath{\protect\lfloor#1\rfloor}}

\begin{document}

\title{\bf Disproof of the List Hadwiger Conjecture\thanks{MSC: graph
    minors 05C83, graph coloring 05C15}}

\author{J\'anos Bar\'at\footnote{School of Mathematical Sciences,
    Monash University, VIC 3800, Australia
    (\texttt{janos.barat@monash.edu}). Research supported by OTKA
    Grant PD~75837, and the J\'anos Bolyai Research Scholarship of the
    Hungarian Academy of Sciences. } 
\and Gwena\"el  Joret\footnote{D\'epartement d'Informatique, Universit\'e Libre de
    Bruxelles, Brussels, Belgium
    (\texttt{gjoret@ulb.ac.be}). Postdoctoral Researcher of the Fonds
    National de la Recherche Scientifique (F.R.S.--FNRS). Supported in
    part by the Actions de Recherche Concert\'ees (ARC) fund of the
    Communaut\'e fran\c{c}aise de Belgique. Also supported by an
    Endeavour Fellowship from the Australian Government.}  \and
  David~R.~Wood\footnote{Department of Mathematics and Statistics, The
    University of Melbourne, Melbourne, Australia
    (\texttt{woodd@unimelb.edu.au}). Supported by a QEII Research
    Fellowship from the Australian Research Council.}}

\maketitle

\begin{abstract}
  The List Hadwiger Conjecture asserts that every $K_t$-minor-free
  graph is $t$-choosable. We disprove this conjecture by constructing
  a $K_{3t+2}$-minor-free graph that is not $4t$-choosable for every
  integer $t\geq 1$.
\end{abstract}

\section{Introduction}
\label{Intro}

In 1943, \citet{Hadwiger43} made the following conjecture, which is
widely considered to be one of the most important open problems in
graph theory; see \citep{Toft-HadwigerSurvey96} for a
survey\footnote{See \citep{Diestel4} for undefined graph-theoretic
  terminology. Let $[a,b]:=\{a,a+1,\dots,b\}$.}.

\begin{HC}
  Every $K_t$-minor-free graph is $(t-1)$-colourable.
\end{HC}

The Hadwiger Conjecture holds for $t\leq 6$ (see
\citep{Hadwiger43,Dirac52,Wagner37,RSST97,RST-Comb93}) and is open for
$t\geq 7$.  In fact, the following more general conjecture is open.

\begin{WHC}
  Every $K_t$-minor-free graph is $ct$-colourable for some constant
  $c\geq 1$.
\end{WHC}

It is natural to consider analogous conjectures for list
colourings\footnote{A \emph{list-assignment} of a graph $G$ is a
  function $L$ that assigns to each vertex $v$ of $G$ a set $L(v)$ of
  colours. $G$ is \emph{$L$-colourable} if there is a colouring of $G$
  such that the colour assigned to each vertex $v$ is in $L(v)$. $G$
  is \emph{$k$-choosable} if $G$ is $L$-colourable for every
  list-assignment $L$ with $|L(v)|\geq k$ for each vertex $v$ of
  $G$. The \emph{choice number} of $G$ is the minimum integer $k$ such
  that $G$ is $k$-choosable. If $G$ is $k$-choosable then $G$ is also
  $k$-colourable---just use the same set of $k$ colours for each
  vertex. Thus the choice number of $G$ is at least the chromatic
  number of $G$. See \citep{Woodall-ListColouringSurvey} for a survey
  on list colouring.}. First, consider the choosability of planar
graphs.  \citet{ERT80} conjectured that some planar graph is not
4-choosable, and that every planar graph is 5-choosable. The first
conjecture was verified by \citet{Voigt-DM93} and the second by
\citet{Thomassen-JCTB94}.  Incidentally, \citet{Boro93} asked whether
every $K_t$-minor-free graph is $(t-1)$-choosable, which is true for
$t\leq 4$ but false for $t=5$ by Voigt's example. The following
natural conjecture arises (see
\citep{KawaMohar-JCTB07,Wood-Contractibility}).

\begin{LHC} 
  Every $K_t$-minor-free graph is $t$-choosable.
\end{LHC}

The List Hadwiger Conjecture holds for $t\leq 5$ (see
\citep{Skrekovski-DM98,HMS-DM08,WL10}). Again the following more
general conjecture is open.

\begin{WLHC} 
  Every $K_t$-minor-free graph is $ct$-choosable for some constant
  $c\geq 1$.
\end{WLHC}

In this paper we disprove the List Hadwiger Conjecture for $t\geq 8$,
and prove that $c\geq\frac{4}{3}$ in the Weak List Hadwiger
Conjecture.

\begin{theorem}
  \label{thm:Main}
  For every integer $t\geq 1$,\\[-5ex]
  \begin{enumerate}
  \item[\textup{(a)}] there is a $K_{3t+2}$-minor-free graph that is
    not $4t$-choosable.\\[-5ex]
  \item[\textup{(b)}] there is a $K_{3t+1}$-minor-free graph that is
    not $(4t-2)$-choosable,\\[-5ex]
  \item[\textup{(c)}] there is a $K_{3t}$-minor-free graph that is not
    $(4t-3)$-choosable.
  \end{enumerate}
\end{theorem}

Before proving Theorem~\ref{thm:Main}, note that adding a dominant
vertex to a graph does not necessarily increase the choice number (as
it does for the chromatic number). For example, $K_{3,3}$ is
3-choosable but not 2-choosable. Adding one dominant vertex to
$K_{3,3}$ gives $K_{1,3,3}$, which again is 3-choosable
\citep{Ohba04}. In fact, this property holds for infinitely many
complete bipartite graphs \citep{Ohba04}; also see \citep{Shen08}.

\section{Proof of Theorem~\ref{thm:Main}}

Let $G_1$ and $G_2$ be graphs, and let $S_i$ be a $k$-clique in each
$G_i$.  Let $G$ be a graph obtained from the disjoint union of $G_1$
and $G_2$ by pairing the vertices in $S_1$ and $S_2$ and identifying
each pair.  Then $G$ is said to be obtained by \emph{pasting} $G_1$
and $G_2$ on $S_1$ and $S_2$.
% Let $\M_t$ be the class of $K_t$-minor-free graphs.
The following lemma is well known.

\begin{lemma}
  \label{Pasting}
  Let $G_1$ and $G_2$ be $K_t$-minor-free graphs.
  % in $\M_t$.
  Let $S_i$ be a $k$-clique in each $G_i$.  Let $G$ be a pasting of
  $G_1$ and $G_2$ on $S_1$ and $S_2$.  Then
  % $G\in\M_t$.
  $G$ is $K_t$-minor-free.
\end{lemma}

\begin{proof}
  Suppose on the contrary that $K_{t+1}$ is a minor of $G$. Let
  $X_1,\dots,X_{t+1}$ be the corresponding branch sets.  If some $X_i$
  does not intersect $G_1$ and some $X_j$ does not intersect $G_2$,
  then no edge joins $X_i$ and $X_j$, which is a contradiction. Thus,
  without loss of generality, each $X_i$ intersects $G_1$. Let
  $X'_i:=G_1[X_i]$. Since $S_1$ is a clique, $X'_i$ is connected. Thus
  $X'_1,\dots,X'_{t+1}$ are the branch sets of a $K_{t+1}$-minor in
  $G_1$.  This contradiction proves that $G$ is
  $K_t$-minor-free. %$G\in\M_t$.
\end{proof}

Let $K_{r\times 2}$ be the complete $r$-partite graph with $r$ colour
classes of size $2$. Let $K_{1,r\times 2}$ be the complete
$(r+1)$-partite graph with $r$ colour classes of size $2$ and one
colour class of size $1$. That is, $K_{r\times2}$ and $K_{1,r\times2}$
are respectively obtained from $K_{2r}$ and $K_{2r+1}$ by deleting a
matching of $r$ edges. The following lemma will be useful.

\begin{lemma}[\citep{Ivanco88,Wood-GC07}]
  \label{MinusMatching}
  $K_{r\times 2}$ is $K_{\floor{3r/2}+1}$-minor-free, and
  $K_{1,r\times 2}$ is $K_{\floor{3r/2}+2}$-minor-free.
\end{lemma}

% We now prove a simple result that is weaker than
% Theorem~\ref{thm:Main} but still strong enough to show that
% $c\geq\frac{4}{3}$ in the Weak List Hadwiger Conjecture.

\begin{proof}[Proof of Theorem~\ref{thm:Main}]
  Our goal is to construct a $K_p$-minor-free graph and a
  non-achievable list assignment with $q$ colours per vertex, where
  the integers $p$, $q$ and $r$ and a graph $H$ are defined in the
  following table.  Let $\{v_1w_1,\dots,v_rw_r\}$ be the deleted
  matching in $H$. By Lemma~\ref{MinusMatching}, the calculation in
  the table shows that $H$ is $K_p$-minor-free.

  \begin{center}
    \begin{tabular}{c|ccccl}
      \hline
      case& $p$ &$q$ & $r$ & $H$ \\\hline
      (a)& $3t+2$& $4t$& $2t+1$ & $K_{r\times2}$ & $\floor{\frac{3}{2}r}+1=3t+2=p$\\
      (b)& $3t+1$& $4t-2$& $2t$ & $K_{r\times2}$ & $\floor{\frac{3}{2}r}+1=3t+1=p$\\
      (c)& $3t$     & $4t-3$& $2t-1$ & $K_{1,r\times2}$ & $\floor{\frac{3}{2}r}+2=3t=p$\\\hline
    \end{tabular}
  \end{center}
  \bigskip

  For each vector $(c_1,\dots,c_r)\in[1,q]^r$, let $H(c_1,\dots,c_r)$
  be a copy of $H$ with the following list assignment. For each
  $i\in[1,r]$, let $L(w_i):=[1,q+1]\setminus\{c_i\}$. Let
  $L(u):=[1,q]$ for each remaining vertex $u$.  There are $q+1$
  colours in total, and $|V(H)|=q+2$.  Thus in every $L$-colouring of
  $H$, two non-adjacent vertices receive the same colour. That is,
  $\col(v_i)=\col(w_i)$ for some $i\in[1,r]$. Since each $c_i\not\in
  L(w_i)$, it is not the case that each vertex $v_i$ is coloured
  $c_i$.

  Let $G$ be the graph obtained by pasting all the graphs
  $H(c_1,\dots,c_r)$, where $(c_1,\dots,c_r)\in[1,q]^r$, on the clique
  $\{v_1,\dots,v_r\}$. The list assignment $L$ is well defined for $G$
  since $L(v_i)=[1,q]$. By Lemma~\ref{Pasting}, $G$ is
  $K_p$-minor-free. Suppose that $G$ is $L$-colourable.  Let $c_i$ be
  the colour assigned to each vertex $v_i$. Thus $c_i\in
  L(v_i)=[1,q]$. Hence, as proved above, the copy $H(c_1,\dots,c_r)$
  is not $L$-colourable. This contradiction proves that $G$ is not
  $L$-colourable. Each vertex of $G$ has a list of $q$ colours in
  $L$. Therefore $G$ is not $q$-choosable. (It is easily seen that $G$
  is $q$-degenerate\footnote{A graph is $d$-degenerate if every
    subgraph has minimum degree at most $d$. Clearly every
    $d$-degenerate graph is $(d+1)$-choosable.}, implying $G$ is
  $(q+1)$-choosable.)\
\end{proof}

Note that this proof was inspired by the construction of a
non-4-choosable planar graph by \citet{Mirz96}.

\section{Conclusion}

Theorem~\ref{thm:Main} disproves the List Hadwiger
Conjecture. However, list colourings remain a viable approach for
attacking Hadwiger's Conjecture. Indeed, list colourings provide
potential routes around some of the known obstacles, such as large
minimum degree, and lack of exact structure theorems; see
\citep{WL10,KawaMohar-JCTB07,Wood-Contractibility,KawaReed-STOC09}.

The following table gives the best known lower and upper bounds on the
maximum choice number of $K_t$-minor-free graphs. Each lower bound is
a special case of Theorem~\ref{thm:Main}. Each upper bound (except
$t=5$) follows from the following degeneracy results.  Every
$K_3$-minor-free graph (that is, every forest) is 1-degenerate.
\citet{Dirac64} proved that every $K_4$-minor-free graph is
2-degenerate. \citet{Mader68} proved that for $t\leq 7$, every
$K_t$-minor-free graph is $(2t-5)$-degenerate. \citet{Jorg94} and
\citet{ST06} proved the same result for $t=8$ and $t=9$
respectively. \citet{Song04} proved that every $K_{10}$-minor-free
graph is 21-degenerate, and that every $K_{11}$-minor-free graph is
25-degenerate. In general, \citet{Kostochka82,Kostochka84} and
\citet{Thomason84,Thomason01} independently proved that every
$K_t$-minor-free graph is $\mathcal{O}(t\sqrt{\log t})$-degnerate.

\begin{center}
  \begin{tabular}{c|ccccccccccccccc}
    \hline
    $t$               & 3 & 4 & 5 & 6 & 7   & 8 & 9 & 10 & 11& $\cdots$ &$t$\\
    \hline
    lower bound & 2 & 3 & 5 & 6 & 7   & 9 & 10 & 11 & 13 & $\cdots$ &$\frac{4}{3}t-c$\\
    upper bound &2  & 3 & 5 & 8 & 10 & 12 & 14 & 22 & 26 &  $\cdots$ &  $\mathcal{O}(t\sqrt{\log t})$\\
    \hline
  \end{tabular}
\end{center}
\bigskip

The following immediate open problems arise:
\begin{itemize}
\item Is every $K_6$-minor-free graph $7$-choosable?
\item Is every $K_6$-minor-free graph $6$-degenerate?
\item Is every $K_6$-minor-free graph $6$-choosable?
\end{itemize}

\subsection*{Acknowledgements}

Thanks to Louis Esperet for stimulating discussions.

%%%%%%%%%%%%%%%%%%%%%%%%%%%%%%%%%%%%%%%%%%%%%%%%%%%%%%%%%%%%%%%%%%
% \bibliographystyle{myNatbibStyle}
% \bibliography{myBibliography,myConferences}
\renewcommand{\baselinestretch}{1}

\def\soft#1{\leavevmode\setbox0=\hbox{h}\dimen7=\ht0\advance \dimen7
  by-1ex\relax\if t#1\relax\rlap{\raise.6\dimen7
    \hbox{\kern.3ex\char'47}}#1\relax\else\if T#1\relax
  \rlap{\raise.5\dimen7\hbox{\kern1.3ex\char'47}}#1\relax \else\if
  d#1\relax\rlap{\raise.5\dimen7\hbox{\kern.9ex
      \char'47}}#1\relax\else\if D#1\relax\rlap{\raise.5\dimen7
    \hbox{\kern1.4ex\char'47}}#1\relax\else\if l#1\relax
  \rlap{\raise.5\dimen7\hbox{\kern.4ex\char'47}}#1\relax \else\if
  L#1\relax\rlap{\raise.5\dimen7\hbox{\kern.7ex
      \char'47}}#1\relax\else\message{accent \string\soft \space #1
    not defined!}#1\relax\fi\fi\fi\fi\fi\fi}

\end{document}